\documentclass[11pt]{amsart}
\usepackage{geometry,amsthm,graphicx,float,enumerate}
\usepackage{amsmath}
\usepackage[makeroom]{cancel}
\usepackage{scalefnt}
\usepackage{mathrsfs}
\usepackage{color}

\usepackage{amssymb}
\restylefloat{table}
\restylefloat{figure}

\swapnumbers
\newtheorem{thm}{Theorem}[section]

\newtheorem{cor}[thm]{Corollary}

\newtheorem{lemma}[thm]{Lemma}
\newtheorem{prop}[thm]{Proposition}

\theoremstyle{definition}
\newtheorem{defn}[thm]{Definition}
\theoremstyle{remark}

\newtheorem{ex}[thm]{Example}

\newcommand{\bb}[1]{\mathbb{#1}}

\newcommand{\Zn}[1]{\mathbb Z_{#1}}
\newcommand{\CPk}[1]{\mathbb C \mathbb P^{#1}  }

\def\UNFs{\Omega_{N,K}} 
\def\Gnk{\mu_{N,K}} 

\newcommand\textem[1]{{\em #1}}

\begin{document}

\title[Achieving the orthoplex bound and constructing weighted complex projective 2-designs]{Achieving the orthoplex bound and constructing weighted\\ complex projective 2-designs  with Singer sets
}

\author{Bernhard G. Bodmann}

\author{John Haas}

\begin{abstract}
Equiangular tight frames are examples of Grassmannian line packings for a Hilbert space. More specifically, according to a bound by Welch, they are minimizers for the maximal magnitude occurring among the inner products of all pairs of vectors in a unit-norm frame. 
This paper is dedicated to packings in the regime in which the number of frame vectors precludes the existence of equiangular frames.
The orthoplex bound then serves as an alternative to infer a geometric structure of optimal designs.
We construct frames of unit-norm vectors in $K$-dimensional complex Hilbert spaces
that achieve the orthoplex bound. When $K-1$ is a prime power, we obtain a tight frame with $K^2+1$ vectors and
when $K$ is a prime power, with $K^2+K-1$ vectors. In addition, we show that these frames
form weighted complex projective 2-designs that are useful additions to  maximal equiangular tight frames
and maximal sets of mutually unbiased bases in quantum state tomography.
Our construction is based on Singer's family of difference sets and the related concept of relative difference sets.
\end{abstract}

\maketitle

\section{Introduction}

More than 50 years ago, the problem of determining the maximal number of equiangular lines in each finite-dimensional Hilbert space appeared in the mathematical literature \cite{MR0023530}. Over time, this has received attention from many researchers, because 
according to a bound by Welch \cite{welch:bound}, equiangular lines can provide optimal packings in real or complex projective spaces. 
These packings have diverse applications, including error correction for analog signals \cite{MR1984549,MR2021601,MR2149656,MR2277977,MR2921716}, wireless communications \cite{MR1984549,MR2028016}, phase retrieval \cite{MR2549940,MR3319532} and quantum state tomography \cite{MR2931102,MR2059685,MR2301093,MR2778089}. However, the last years have shown that despite efforts by many researchers, a systematic and feasible construction of conjectured
sets of equiangular lines remains challenging, especially in the complex case \cite{MR1813161,MR2015832,MR2460230,MR2662471,szollHosi2014all}. Moreover, the Welch bound cannot be achieved if the number of lines gets large compared to the dimension of the space.
In the case of a $K$-dimensional real or complex Hilbert space, 
if the number of lines is larger than $K(K+1)/2$ or $K^2$, respectively, then this implies that the best packing cannot be equiangular. 
The lack of a known algebraic or geometric structure in this regime makes
this optimal design problem difficult. If the number of lines is not too large, then the orthoplex
bound~\cite{MR0074013, MR1418961, henkel05} offers an alternative way to construct best packings with specific geometric characteristics. When choosing unit-norm vectors as representatives of the lines, then the orthoplex bound
states that the maximal magnitude among the inner product of pairs of vectors cannot be smaller than $1/\sqrt K$.
Intriguingly, this quantity appears in the definition of mutually unbiased bases that originated in quantum information theory
\cite{MR0115648,MR1943521,MR2363400,MR3098293}. To the best of our knowledge,
the only notable type of configurations that meet the orthoplex bound seems to be obtained with
maximal sets of mutually unbiased bases, which are known to exist in all dimensions that are prime 
powers \cite{MR1943521,MR2363400,MR3098293}. A maximal set of mutually unbiased bases in
a $K$-dimensional complex Hilbert space consists of $K+1$ orthonormal bases. If two vectors belong to different bases,
then the magnitude of their inner product is $1/\sqrt K$. By removing vectors from one basis, one can then 
achieve the orthoplex bound with a total number of vectors between $K^2+1$ and $K(K+1)$.

In addition to achieving optimal packings, another property that is desirable is that of designs.
Maximal sets of mutually unbiased bases and maximal equiangular frames have been shown to
be complex projective 2-designs \cite{1523643}. This makes them useful for quantum state tomography \cite{MR2269701}.
The defining property of a $1$-design 
is that the orthogonal projections onto the lines spanned by the vectors
sum to a multiple of the identity. These vectors are then also called a tight frame. A $2$-design provides, in addition, a resolution of the identity for matrices \cite{MR2931102}. 
Because designs with few vectors are hard to construct, the more general concept of weighted $t$-designs has been introduced and studied \cite{MR2269701}.

In this paper, we construct more examples of lines, or equivalently, unit-norm vectors that saturate the orthoplex bound. 
It is intriguing that
they are obtained by families of $K^2+1$ or $K^2+K-1$ vectors in a $K$-dimensional  complex Hilbert space, thus
in size close to maximal sets of equiangular vectors or maximal sets of mutually unbiased bases. 
In contrast to the unresolved existence of $K^2$ equiangular vectors in each $K$-dimensional complex Hilbert space \cite{MR2662471},  
this construction
requires $K-1$ or $K$ to be a prime power, which provides us with two infinite families of examples, next to the known
maximal sets of mutually unbiased bases.
We construct these examples by adjoining an orthonormal basis and equiangular or near-equiangular tight frames
consisting of flat vectors. The concept of mutual unbiasedness of bases or frames is an essential ingredient in our constructions. 
As a consequence of this geometric property, we also 
obtain weighted complex projective $2$-designs. For linear quantum state tomography, such designs have been shown to be an optimal choice 
among all measurements with a fixed number of outcomes
 \cite[Sections 5 and 6]{MR2269701}.   A byproduct of our work is the construction of mutually unbiased sets of tight frames,
 which seem to be previously unknown in the literature. Another minor consequence is that our construction provides
 an alternative way to produce optimal packings with a number of vectors ranging between $K^2+1$ and $K^2+K-2$,
 by omitting vectors from a set of size $K^2+K-1$, in close analogy with the construction of optimal packings based on subselection from
 a maximal set of mutually unbiased bases.
 
 The fundamental technique we use in our construction is closely related to difference sets \cite{MR1501951,MR1735401,MR1440858}, which 
 have been used previously to construct equiangular lines that meet the Welch bound \cite{MR2235693}.
 We augment the difference set construction and also consider relative difference sets \cite{MR2057609}
 that have appeared in the construction of mutually unbiased bases \cite{MR2460230}.

This paper is organized as follows. In Section~\ref{sec:prelim}, we fix notation and terminology.
Section~\ref{sec:orthoplex} is dedicated to the construction of vectors that achieve the orthoplex-bound
with the help of difference sets or relative difference sets. We conclude by showing that these 
constructions provide weighted complex projective $2$-designs that have relevance for quantum state tomography.

\section{Preliminaries}\label{sec:prelim}
This paper is concerned with a special type of frame for finite-dimensional complex Hilbert spaces. After choosing 
an orthonormal basis $\{e_i\}_{i=1}^K$ for a $K$-dimensional Hilbert space, $K \in \mathbb N$, we identify it
with $\mathbb C^K$. A \textem{frame} $\mathcal F = \{f_j\}_{j \in J}$ for ${\mathbb C}^K$ is an at most countable family of vectors such that there are
constants $A,B>0$ for which the frame bounds
$$
  A \|x \|^2 \le \sum_{j\in J} |\langle x, f_j \rangle |^2 \le B \|x\|^2
$$
hold for each $x \in \mathbb C^K$.

A frame $\{f_j\}_{j \in J}$ is called \textem{$A$-tight} if we can choose $A=B$ in the above bounds.
A frame $\{f_j\}_{j \in J}$ is called \textem{unit-norm} if each frame vector has norm $\|f_j\|=1$.
We call a vector $x \in \mathbb C^K$ \textem{flat} if $|\langle x, e_i\rangle | = \|x\|/\sqrt{K}$ for each $i\in \{1, 2, \dots, K\}$, and we call a frame ${\mathcal F} = \{f_j\}_{j \in J}$ \textem{flat} if each
frame vector $f_j\in \mathcal F $ has this property.

Let $\UNFs$ denote the space of unit norm frames consisting of $N$ vectors in $\mathbb C^K$.
Because $N=\sum_{j \in J} \|f_j\|^2 = \sum_{i=1}^K\sum_{j\in J} |\langle e_i, f_j \rangle|^2  \le B K$, the size of such a frame is bounded by $N \le BK$. 
For the set of unit-norm frames, we let 
$$\Gnk = \inf\limits_{{\mathcal F} \in \UNFs} \max\limits_{j \neq l} |\langle f_j, f_l \rangle | \, ,$$ 
i.e.\ \textem{the Grassmannian constant for the pair $(N,K)$}.   
We also write $\mu(\mathcal F) = \max\limits_{j \neq l} |\langle f_j, f_l \rangle |$.
This constant is bounded by Welch's inequality \cite{welch:bound},
$$
  \Gnk \ge \sqrt{\frac{N-K}{K(N-1)}} \, ,
$$
which is saturated if and only if there is a  $N/K$-tight unit-norm frame that is equiangular, meaning $|\langle f_j, f_l \rangle | = \mu_{N,K}$
for each $j\ne l$. It can be shown that such frames only exist if $N \le K^2$ \cite{delsarte1975bounds, MR0404174}, so the bound can be improved if
$N > K^2$. 

We recall from \cite{MR0074013, MR1418961, henkel05} that for $K^2+1 \leq N$ we have the \textem{orthoplex} (or \textem{Rankin}) \textem{bound}.

\begin{thm}[\cite{MR0074013, MR1418961, henkel05}]
If $K^2+1 \leq N$, then
$$
\Gnk \geq \frac{1}{\sqrt{K}} \, .
$$
\end{thm}
\begin{proof}
Let $\mathcal F = \{f_j\}_{j \in \Zn{N}} \in \Omega_{N,K}$.  The frame vectors of $\mathcal F$ can be  embedded into the traceless subspace of the 
real Hilbert space of $K \times K$ Hermitians  via the mapping
$$ f_j \mapsto T_j \equiv f_j \otimes f_j^* - \frac{1}{K }I_K,$$
whose range has then dimension at most $K^2-1$.  Moreover, these embedded vectors all lie on a sphere of radius $\sqrt{\frac{K-1}{K}}$, because the Hilbert-Schmidt norm gives $\|T_j\|_{H.S.}^2 = \| f_j \otimes f_j^* - \frac{1}{K }I_K\|_{H.S.}^2 = \frac{K-1}{K}$ for every $j \in \Zn{N}$. The 
embdedding is distance preserving when the vectors $f_i$ and $f_j$ are interpreted as representatives of points $[f_i]$ and $[f_j]$
in complex projective space,
equipped with the chordal distance $d_c$,  because the multiple of the identity cancels when computing  for $j \neq l$
\begin{align*}
\| T_j - T_l \|_{H.S.}^2 = \| f_j \otimes f_j^* - f_l \otimes f_l^*\|_{H.S.}^2 = d_c([f_i],[f_j])^2 \, .
\end{align*}
By identifying $f_j \otimes f_j^* - \frac{1}{K }I_K$ and $f_l \otimes f_l^* - \frac{1}{K }I_K$ with vectors in $\mathbb R^{K^2-1}$ on a sphere of radius $\sqrt{\frac{K-1}{K}}$, we obtain
\begin{align*}
\|T_j - T_l \|_{H.S.}^2  
&= \frac{2(K-1)}{K} - 2 \langle T_j , T_l \rangle_{H.S.} \, .
\end{align*}
On the other hand, expressing the chordal distance in terms of the squared inner product 
$$
 d_c([f_i],[f_j])^2= 2\left(1- tr( f_j \otimes f_j^* f_l \otimes f_l^* )\right) 
= 2\left(1- |\langle f_j, f_l \rangle|^2 \right)  
$$
gives the identity 
$$
 |\langle f_ j, f_l \rangle|^2 = \frac{1}{K} + \frac{K-1}{K} \cos \phi_{j,l},
$$
between the inner product of $f_j$ and $f_l$ and  the angle $\phi_{j,l}$ between $T_j$ and $T_l$.
As  shown in \cite{MR0074013}, the best possible possible packing of $d+1$ points on a sphere in $\mathbb R^d$ cannot be improved beyond the packing corresponding to the vertices of an $l^1$-ball, an orthoplex. The claim follows because
in that case all pairs $T_j$ and $T_l$ are orthogonal, which implies that the optimal magnitude of the inner product
between two vectors $f_j$ and $f_l$ is $|\langle f_j, f_l \rangle | = 1/\sqrt K$.
\end{proof}

\begin{defn}
Let $K^2+1 \leq N$.  If $\mathcal F \in \UNFs$ and $\max\limits_{j,l \in \Zn{N}, j \neq l} |\langle f_j, f_l\rangle |= 1/\sqrt K$, then we call $\mathcal F$ an \textem{orthoplectic Grassmannian frame} or \textem{OGF}.
\end{defn}

In analogy with the Welch bound, orthoplectic Grassmannian frames can only exist if $N \le 2(K^2-1)$, because
when the upper bound is reached,
the vertices of the orthoplex are all occupied \cite{MR0074013}.

A known type of frame that saturates the orthoplex bound is a maximal set of mutually unbiased bases. This means, the frame 
for ${\mathbb C}^K$ consists of 
vectors $\{e_i^{(k)}: 1 \le i \le K, 1 \le k \le K+1\}$ such that for a fixed index $k$, $\{e_i^{(k)}: 1 \le i \le K\}$ forms an
orthonormal basis, and the inner products between any two vectors belonging to two bases indexed by $k, k'$,
$k \ne k'$, has magnitude  $|\langle e_{i}^{(k)}, e_{i'}^{(k')} \rangle | = \frac{1}{\sqrt K}$.

\begin{ex}\label{prop:dchirps}
Let $K$ be a prime number, and $\omega$ a primitive $K$-th root of unity. Denote
the canonical basis of $\bb C^K$ by $\{e^{(1)}_i\}_{i=1}^K$,  
then for $k\in \{2, 3, \dots K+1\}$,
$$
    e^{(k)}_i = \frac{1}{\sqrt K} \sum_{l=1}^K \omega^{-(k-1)l^2+il} e^{(1)}_l    \, 
$$ 
defines 
together with the canonical basis a family of $K+1$ mutually unbiased bases
called the discrete chirps. 
\end{ex}

We also define a slightly more general notion of unbiasedness between two unit-norm tight frames.

\begin{defn}
Two unit-norm tight frames $\mathcal F=\{f_j\}_{j=1}^N$ and $\mathcal F'=\{f'_l\}_{l=1}^M$ for $\mathbb C^K$ are called mutually
unbiased if for each pair of indices $j \in \{1, 2, \dots, N\}$ and $l \in \{1, 2, \dots, M\}$, the inner product of the corresponding frame vectors from $\mathcal F$ and $\mathcal F'$ has magnitude
$|\langle f_j, f'_l \rangle | = 1/\sqrt K$. 
\end{defn}

This paper is dedicated to the construction of other types of OGFs.
By examining the behavior of the gradient descent for a frame potential that is a smooth perturbation of  $\mu(\mathcal F)$ \cite{bgb15}, we discovered
the following OGF in $\Omega_{5,2}$. 

\begin{ex}\label{ex_5_2}
 Let $\{e_i\}_{i=1}^2$ be the canonical basis in ${\mathbb C}^2$, $\omega=e^{2\pi i/3}$ 
 and let $h_j = a (e_1 + \omega^j e_2)$ with $a = \frac{1}{\sqrt 2}$, then
 $\mathcal F = \{e_1, e_2, h_1, h_2, h_3\}$ is an OGF because the inner product between
 the two orthonormal basis vectors is zero, the inner products between $h_j$ and $h_l$
 with $j \ne l$ have magnitude $|\langle h_j, h_l \rangle | = \frac1 2 $ and the inner product between
 any $e_i$ and $h_j$ has magnitude $|\langle e_i, h_j \rangle | = \frac{1}{\sqrt 2} $.
\end{ex}

The main ingredient for this construction is a  set of unit-norm flat vectors with pairwise inner products of sufficiently small magnitudes.  If such a set is of an appropriate cardinality, then its union with an orthonormal basis yields an OGF.  We state this formally in the following theorem.

\begin{thm}\label{thmain}
Let $N = M +K$, where $M,K \in \mathbb N$ with $K^2+1 \leq N$ and suppose that ${\mathcal F}' = \{h_j\}_{j \in \Zn{M}} \subset \mathbb C^K$ is a set of flat unit vectors.  If 
$\max\limits_{j,l \in \Zn{M}, j \neq l} | \langle h_j, h_l \rangle | \leq \frac{1}{\sqrt{K}}$, then the set $\mathcal F = \{e_j\}_{j \in \Zn{K}}\cup {\mathcal F}'$ forms an orthoplectic Grassmannian frame, where  $\{e_j\}_{j \in \Zn{K}}$ denotes the canonical orthonormal basis for $\mathbb C^K$.  Furthermore, if ${\mathcal F}'$ forms a tight frame for $\mathbb C^K$, then $\mathcal F$ is also tight.
\end{thm}

\begin{proof}
Since $\{e_j\}_{j \in \Zn{K}}$ spans $\mathbb C^K$ and ${\mathcal F}'$ consists of unit norm vectors, it follows that $\mathcal F \in \Omega_{N,K}$.  Among all pairwise inner products between the vectors of $\mathcal F$, the maximal magnitude occurs when one of the vectors is from $\{e_j\}_{j \in \Zn{K}}$ and the other is from ${\mathcal F}'$ and this value is precisely $\frac{1}{\sqrt{K}}$.  By hypothesis, $K^2+1 \leq N$, so $\mathcal F$ is an OGF for $\mathbb C^K$.  The statement on tightness follows because the union of two tight frames $\mathbb C^K$ is again a tight frame for $\mathbb C^K$.
\end{proof}

%
%
In this paper, we pursue the idea of using cyclic frames to achieve this condition. 

\begin{defn}\label{cyclicdef}
Let  ${\mathbb Z}_M = \{0, 1, \dots, M-1\}$ denote the additive group of integers modulo $M \in \mathbb N$.
Let  $K \in \mathbb N$ and $\{n_1, n_2, ..., n_K\} \subset {\mathbb Z}_M$.  A frame $\mathcal F =\{ h_j \}_{j \in \Zn{M}}$, where 
$h_j = \frac{1}{\sqrt{K}}  \sum_{l=1}^K
e^{2 \pi i j n_l /M} e_l$ for all $j \in \Zn{M}$, is called a \textem{cyclic frame for $\mathbb C^K$ generated by the sequence $\{n_1, n_2 ..., n_K\}$}.  When the specific choice of $\{n_1, n_2 ..., n_K\}$ is not important, we call $\mathcal F$ a \textem{cyclic frame for $\mathbb C^K$} and, when the dimension of the underlying vector space is obvious, just a \textem{cyclic frame}.
\end{defn}

\section{Families of orthoplex-bound achieving Grassmannian frames}\label{sec:orthoplex}

Cyclic frames are by definition flat with respect to the canonical orthonormal basis, so if ${\mathcal F}'$ is a cyclic frame 
indexed by ${\mathbb Z}_M$,
then it can be augmented by adjoining the 
orthonormal basis vectors to form a frame $\mathcal F$ of $M+K$ vectors having $\mu({\mathcal F}) = \max\{\mu({\mathcal F}'), 1/\sqrt K\}$.
In order to achieve the orthoplex bound, we seek sequences $\{n_1, n_2, ..., n_{K} \} \subset  \Zn{M}$, where $K^2+1 \leq M +K$ such that
the cyclic frames they generate satisfy
$$
|\langle h_j, h_0 \rangle | \leq \frac{1}{\sqrt{K}}
$$
for all $j  \ne 0$.

%
In the following section, we provide two families of OGFs, both of which depend on concepts from combinatorial design theory.
In order to prepare this section, we collect some facts on cyclic frames related to the discrete Fourier transform.


\begin{defn}\label{modopdef}
Let $\mathcal F = \{f _j\}_{j \in \Zn{M}}$ be any frame and let $\omega_{M}= e^{2 \pi i /M}$.  For $\xi \in \Zn{M}$, we define the $\xi$th modulation operator for $\mathcal F$ by
$$
X_{\xi} = \sum\limits_{j \in \Zn{M}} \omega_{M}^{\xi j } f_j \otimes f_j^*
.$$
\end{defn}

We begin by recording a simple computation that relates the inner products between frame vectors and the Hilbert-Schmidt inner products between the modulation operators.

\begin{lemma}\label{dotFTval}
If $\mathcal F = \{ f_j \}_{j \in \Zn{M}}$ is any frame consisting of $M$ vectors over $\mathbb F^K$, then 
$${M}^{2} | \langle f_a, f_b \rangle |^2 
= \sum\limits_{\xi, \eta \in \Zn{M}} \omega_{M}^{b \eta-a \xi} \langle X_{\xi}, X_{\eta} \rangle_{H.S.}$$
 for all $a, b \in \Zn{M}$.
\end{lemma}

\begin{proof}
This follows from a straightforward application of the Fourier inversion formula and the fact that 
$$
| \langle f_a, f_b \rangle |^2  = \langle f_a \otimes f_a^*, f_b \otimes f_b^* \rangle_{H.S.}.$$
\end{proof}

Next, we compute the entries of the modulation operators of cyclic frames with respect to the canonical basis.

\begin{prop}\label{cyclicmodcomp}
If  $\mathcal F = \{ h_j \}_{j \in  \Zn{M}}$ is a cyclic frame generated by the sequence $\{n_1, n_2 ..., n_K\}$, then 
the entries of the associated modulation operators $\{X_\xi\}_{\xi=1}^{M-1}$ with respect to the canonical basis are
given by
$$(X_\xi)_{a,b} = 
\left\{ \begin{array}{cc} 
 \frac{M}{K}, & n_b - n_a = \xi  \\  
 0, &  \text{otherwise}
  \end{array} \right. ,$$ 
for every $a, b, \xi \in \Zn{M}$.
\end{prop}

\begin{proof}
If $j , \xi \in \Zn{M}$, then 
the $(a,b)$-entry of $h_j \otimes h_j^*$ is 
$$
(h_j \otimes h_j^*)_{a,b} = \frac{1}{K} \omega_{M}^{j (n_a - n_b)}.
$$
so, by Definition~\ref{modopdef}, the $(a,b)$-entry of $X_\xi$ is 
$$
(X_\xi)_{a,b} 
=
\left(        \sum_{j \in \Zn{M}} \omega_{M}^{j \xi} h_j \otimes h_j^*              \right)_{a,b}  
=
 \sum_{j \in \Zn{M}} \omega_{M}^{j \xi} ( h_j \otimes h_j^*    )_{a,b}
=
 \frac{1}{K} \sum_{j \in \Zn{M}}      \omega_M^{j (\xi + n_a - n_b)}.   
$$
If $n_b - n_a \neq \xi$, the entry vanishes as a summation of consecutive powers of a root of unity;  otherwise, $\omega_M^{j (\xi + n_a - n_b)}=1$ and the claim follows.
\end{proof}

This shows that the non-zero entries of the modulation operator $X_\xi$ are indexed by
 subset $\{(a,b): n_b - n_a = \xi\}$ of  ${\mathbb Z}_K \times {\mathbb Z}_K$.
 Thus, if $\xi \ne \zeta$, then
the $(a,b)$ entries of $X_\xi$ and $X_\zeta$ cannot both be non-zero.  

This implies that the 
modulation operators of cyclic frames are Hilbert-Schmidt orthogonal.

\begin{cor}\label{cyclic_ortho}
If $\mathcal F = \{ h_j \}_{j \in \Zn{M}}$ is a cyclic frame, then
$$ \langle X_\xi,  X_\zeta \rangle_{H.S.} = 0$$
for all $\xi, \zeta \in \Zn{M}$ with $\xi \neq \zeta$.
\end{cor}


Combining Lemma~\ref{dotFTval} with Corollary~\ref{cyclic_ortho}, we obtain the following simplified relationship between the inner products between frame vectors and the Hilbert-Schmidt inner products between the modulation operators of cyclic frames.

\begin{cor}\label{FTvalsimp}
If $\mathcal F = \{ h_j \}_{j \in \Zn{M}}$ is a cyclic frame for $\mathbb C^K$, then 
$${M}^{2} | \langle h_a, h_0 \rangle |^2 
= \sum\limits_{\xi \in \Zn{M}} \omega_{M}^{a \xi} \| X_{\xi} \|_{H.S.}^2$$
 for all $a \in \Zn{M}$.
\end{cor}

In \cite{MR1943521}, the authors used modulation operators to prove that the maximal number of mutually unbiased bases that can exist in $\mathbb F^K$ is $K+1$.  In our case, we use them to show that the maximal magnitude among pairwise inner products from a cyclic frame generated by a suitable choice of $\{n_1, n_2, \dots, n_K\}$ is $\frac{1}{\sqrt{K}}$ whenever $K$ or $K-1$ is a prime power.

\subsection{A difference set construction: $N=K^2+1$, $K=q+1$, $q$ a prime power}

In this subsection, we extend Example~\ref{ex_5_2} to an infinite family of OGFs.  This relies on the notion of difference sets.

\begin{defn}\label{defdiffset}
A \textem{$(M,K, \lambda)$-difference set for $\Zn{M}$} is a subset of distinct elements $\{ n_1, n_2, ..., n_K\} \subset \Zn{M}$ such that every nonzero element $x \in \Zn{M}$ can be expressed as $x = n_j - n_l$ in exactly $\lambda$ ways, where $\lambda$ is a positive integer.
\end{defn}

It is well known that cyclic frames generated by difference sets form equiangular tight frames \cite{MR2235693, MR2277977}.  Because equiangular tight frames minimize the maximal magnitude among pairwise inner products whenever they exist, difference sets are a natural starting point for our construction of orthoplectic Grassmannian frames.

The Singer family of difference sets is parametrized by
$n \in \mathbb N$ and a prime power $q$.

\begin{thm}[\cite{MR1735401,MR1440858}] Let $q$ be a prime power, $n \in \mathbb N$, 
$M= \frac{q^{n+2}-1}{q-1}$, $K=\frac{q^{n+1}-1}{q-1}$, $\lambda =  \frac{q^{n}-1}{q-1}$
and define the map $tr_{q^{n+2}/q}: {\mathbb Z}_{q^{n+2}} \to 
{\mathbb Z}_{q}$ by $tr_{q^{n+2}/q}(x) = \sum_{i=0}^{n+1} x^{q^i}$.
Let $\alpha$ be a primitive root of the multiplicative group of ${\mathbb Z}_{q^{n+2}}$,
 then
 $\{i: 0 \le i <(q^{n+2}-1)/(q-1), tr_{q^{n+2}/q}(\alpha^i) = 0\}$ forms a $(M,K,\lambda)$-difference set.
\end{thm}

In the case $n=1$, we have $M=q^2 + q +1$, $K=q+1$ and $\lambda=1$.
Consequently, 
$$M + K = (q^2 +q +1) + (q+1) = (q+1)^2 + 1 = K^2 +1.$$

\begin{thm}\label{thmdiffsetsgiveortho}
Let $M = K^2-K+1$.  If $S = \{n_1, n_2, ..., n_K\} \subset {\mathbb Z}_{M}$ is a $(M,K,1)$-difference set for $\Zn{M}$ 
and $\mathcal F = \{h_j\}_{j \in \Zn{M}}$ is the cyclic frame generated by $S$, then for each $j \ne l$,
$$
|\langle h_j, h_l \rangle| = \frac{\sqrt{K-1}}{K} .
$$
\end{thm}

\begin{proof}
By Proposition~\ref{cyclicmodcomp}, we have
$$\|(X_0) \|_{H.S.}^2 = \frac{M^2}{K}$$
and from $\lambda=1$
$$\|(X_\xi) \|_{H.S.}^2 = 
 \frac{M^2}{K^2} 
 $$ 
for every $\xi \in \Zn{M}$ with $\xi \neq 0$.  Combining this with Corollary~\ref{FTvalsimp}, we have 
$$
|\langle h_j, h_l \rangle |^2 =  |\langle h_0, h_{l-j} \rangle |^2  = 
\frac{1}{K^2}
\sum_{\xi=1}^{M-1}
 %
 \omega_{M}^{(l-j) \xi} + \frac{1}{K} = \frac 1 K - \frac{1}{K^2} \, .
$$
for all $j, l \in \Zn{M}$ with $j \neq l$, where $\omega_{M} = e^{2 \pi i / {M} }$.
\end{proof}

\begin{cor}
Let $q$ be a prime power and let ${\mathcal F}'= \{h_j\}_{j \in {\mathbb Z}_{M}}$ be a cyclic frame  generated by 
a $(M,K,1)$-difference set with $M=q^2 + q +1$ and $K=q+1$, then 
$$
|\langle h_j, h_l \rangle | 
= \frac{\sqrt{K-1}}{K} < \frac{1}{\sqrt{K}}
$$
for all $j , l \in \Zn{M}$ with $j \neq l$ and
if $e_i$ is a canonical basis vector, then $| \langle e_i, h_j \rangle | = \frac{1}{\sqrt K}$, so
$ \mathcal F = \{e_i\}_{i=1}^K \cup \{h_j \}_{j=0}^{M-1} $
is a tight OGF in $\Omega_{N,K}$ with $N=K^2+1$, as outlined in Theorem~\ref{thmain}.
\end{cor}

\begin{table}
\begin{tabular}{| l | l |  }
\hline K & Difference set \\ 
\hline 3 & \{0, 1, 3\}\\ 
\hline 4 & \{0, 1, 3, 9\}\\ 
\hline 5 & \{0, 1, 4, 14, 16\}\\
\hline 6 & \{0, 1, 3, 8, 12, 18\}\\
\hline 7 & DNE \\
\hline 8 & \{0, 1, 3, 13, 32, 36, 43, 52\}\\
\hline 9 & \{0, 1, 3, 7, 15, 31, 36, 54, 63\} \\
\hline 10 & DNE\\
\hline \end{tabular}
\caption{Examples of $(M,K,1)$-difference sets constructed by Singer~\cite{MR1501951}
with $M=q^2+q+1$, $K=q+1$ and $\lambda=1$ for lowest values of $q$.}
\end{table}

\subsection{A construction based on relative difference sets: $N=K^2+K-1$, $K=q$, $q$ a prime power}
In this subsection, we use a combinatorial concept  closely related to difference sets to construct another type of OGFs.  

\begin{defn}\label{defsemidiffset}
A subset of distinct elements $S = \{ n_1, n_2, ..., n_K\} \subset \Zn{M}$ with $M=NL$ is a  \textem{$(N,L,K, \lambda)$-relative difference set for $\Zn{M}$} if there exists some ``forbidden'' subgroup $G \subset \Zn{M}$ of size $|G|=L$ such that every element $x \in S\setminus G$ can be expressed as $x = n_j - n_l$ in exactly $\lambda$ ways, where $\lambda$ is some positive integer, and no element of $G$ occurs among the differences $n_j - n_l$, $n \ne l$.
%

In the special case that $S$ is a $(q+1, q-1, q, 1)$-relative difference set, then we say that $S$ is a \textem{picket fence sequence for $\Zn{q^2-1}$}.
\end{defn}

\begin{thm}[\cite{MR1440858}] 
If $q$ is a prime power, $N=\frac{q^{n+1}-1}{q-1}$, 
$L=q-1$, $K=q^{n}$ and $\lambda=q^{n-1}$, $tr_{q^{n+1}/q}$ maps from ${\mathbb Z}_{q^{n+1}}$
to ${\mathbb Z}_q$ by $tr_{q^{n+1}/q}(x) = \sum_{i=0}^{n} x^{q^i}$,
and $\alpha$ is a primitive element of the multiplicative group of ${\mathbb Z}_{q^{n+1}}$,
then the set $\{i \in {\mathbb Z}_{q^{n+1}}: tr_{q^{n+1}/q}(\alpha^i)=1\}$ is
a  $(N,L,K,\lambda)$-relative difference set in $\Zn{NL}$.
\end{thm}

If we choose $n=1$, then we get picket fence sequences.

\begin{ex}
The sequence $\{0,1,3\}$ is a picket fence sequence for $\Zn{8}$ because $1-0=1$, $3-1=2$, $3-0=3$, $0-3=5$, $1-3=6$, and $0-1=7$. This corresponds to $q=3$, $N=4$, $L=2$, $K=3$ and $\lambda=1$.
The condition of the preceding theorem can be verified with the primitive element $\alpha=2$ of ${\mathbb Z}_9$.
\end{ex}

Next, we show that picket fence sequences generate OGFs.


\begin{thm}\label{thmpartdiffsetsgiveortho}
Let $M = K^2-1$.  If $S = \{n_1, n_2, ..., n_K\} \subset {\mathbb Z}_{M}$ is a picket fence sequence for $\Zn{M}$ and $\mathcal F = \{h_j\}_{j \in \Zn{M}}$ is the cyclic frame generated by $S$, then
$$
|\langle h_0, h_a\rangle| =
\left\{
\begin{array}{cc}
1, & a=0 \\
\frac{1}{K}, & a\neq 0 \text{ and } a \equiv 0 \mod (K-1) \\
\frac{1}{\sqrt{K}}, &  a \not\equiv 0 \mod (K-1) 
\end{array}
\right.
.$$
In particular,
$$
\max\limits_{j \ne l} |\langle h_j, h_l \rangle| = \frac{1}{\sqrt K} .
$$
\end{thm}

\begin{proof}
By Definition~\ref{defsemidiffset} and Proposition~\ref{cyclicmodcomp}, we have
$$\|(X_0) \|_{H.S.}^2 = \frac{M^2}{K}$$
and
$$\|(X_\xi) \|_{H.S.}^2 = 
\left\{ \begin{array}{cc} 
 0, & \xi = m (K+1), m \in \Zn{M}  \\  
 \frac{M^2}{K^2}, &  \text{otherwise}
  \end{array} \right. ,$$ 
for every $\xi \in \Zn{M}$ with $\xi \neq 0$.  Combining this with Corollary~\ref{FTvalsimp}, we have 
$$
 |\langle h_0, h_a \rangle |^2  = 
\frac{1}{K^2}
\sum\limits_{\tiny 
\begin{array}{cc} 
\xi \in \Zn{M}, \\ 
\xi \neq m (K+1) \\
 \text{ for } m  \in  \Zn{M} 
  \end{array} } \omega_{M}^{a \xi} + \frac{1}{K}
$$
for all $a \in \Zn{M}$ with $a \neq 0$, where $\omega_{M} = e^{2 \pi i / {M} }$.
By replacing the indices where $\xi = m(K+1)$ and using that $M = K^2 -1$, the summation in the first term can be rewritten
\begin{align*}
\sum\limits_{\tiny 
\begin{array}{cc} 
\xi \in \Zn{M}, \\ 
\xi \neq m (K+1) \\
 \text{ for } m  \in  \Zn{M} 
  \end{array} } \omega_{M}^{a \xi}
  &=
\sum\limits_{\tiny 
\xi \in \Zn{M} } \omega_{M}^{a \xi}
-
\sum\limits_{\tiny m \in \Zn{K-1}} 
\omega_{M}^{a m(K+1)} \\
&=
- \sum\limits_{\tiny m \in \Zn{K-1}} 
e^{\frac{2 \pi i a m}{K-1}},
\\
\end{align*}
where the first term vanished because of the summation of consecutive roots of unity.
If $a \equiv 0 \mod (K-1)$, then 
$$
\sum\limits_{\tiny m \in \Zn{K-1}} 
e^{\frac{2 \pi i a m}{K-1}} = K-1,
  $$
  and
  $$
   |\langle h_0, h_a \rangle | = \frac{1}{K}.
  $$
If $a \not\equiv 0 \mod (K-1)$, then 
$$\sum\limits_{\tiny m \in \Zn{K-1}} 
e^{\frac{2 \pi i a m}{K-1}} = 0$$
due to the summation of consecutive $(K-1)$th roots of unity and, in particular,
 $$
   |\langle h_0, h_a \rangle | = \frac{1}{\sqrt{K}},
  $$
  which completes the proof.
\end{proof}

\begin{cor}
Let $q$ be a prime power and let ${\mathcal F}'= \{h_j\}_{j \in {\mathbb Z}_{M}}$ be a cyclic frame  generated by 
a $(N,L,K,1)$-relative difference set with $N=q+1$, $L=q-1$ and $K=q$, then 
$$
|\langle h_j, h_l \rangle | 
\le \frac{1}{\sqrt K} 
$$
for all $j , l \in \Zn{M}$ with $j \neq l$ and
if $e_i$ is a canonical basis vector, then $| \langle e_i, h_j \rangle | = \frac{1}{\sqrt K}$, so
$ \mathcal F = \{e_i\}_{i=1}^K \cup \{h_j \}_{j=0}^{NL-1} $
is a tight OGF in $\Omega_{K^2+K-1,K}$, as outlined in Theorem~\ref{thmain}.
\end{cor}

As in the case of the maximal set of mutually unbiased bases, we can obtain
more examples of OGFs of smaller sizes by subselection. 

\begin{cor}
By removing between $1$ and $K-2$ vectors from the orthonormal basis in $\mathcal F$ in the preceding corollary, we obtain a subset of $\mathcal F$
with size between $K^2+1$ and $K^2+K-2$ which also saturates the orthoplex bound.
\end{cor}

However, these (strict) subsets do not form a tight OGF, which makes them less interesting.

We used the software package GAP (Groups, Algorithms, Programming) version 4.7.8 to perform an exhaustive search for picket fence sequences.  Part of the results are presented in Table~\ref{tab:picket}.  For each $K$, if a picket fence sequence exists, then we list an example, otherwise, it is marked DNE.

\begin{table}
\begin{tabular}{| l | l |  }
\hline K & Picket Fence Sequence \\ 
\hline 3 & \{0, 1, 3\}\\ 
\hline 4 & \{0, 1, 3, 7\}\\ 
\hline 5 & \{0, 1, 3, 11, 20\}\\
\hline 6 & DNE\\
\hline 7 & \{0, 1, 3, 15, 20, 38, 42\}\\
\hline 8 & \{0, 1, 3, 7, 15, 20, 31, 41\}\\
\hline 9 & \{0, 1, 3, 9, 22, 27, 34, 38, 66\} \\
\hline 10 & DNE\\
\hline \end{tabular}
\caption{Examples of picket fence sequences for $K=q$, in ${\mathbb Z}_{q^2-1}$.}\label{tab:picket}
\end{table}

\section{Weighted complex projective $2$-designs}
In this section, we point out the significance of the orthoplectic Grassmannian frames constructed here for quantum state determination by showing that certain OGFs form weighted complex projective $2$-designs.

The \textem{complex projective space, $\CPk{K-1}$,} 
 is the set of all lines passing through the origin in $\mathbb C^K$.  Every line $[x] \in \CPk{K-1}$ can be represented by a unit vector $x \in \mathbb C^K,$ which yields the rank one projection operator \textem{$\pi(x) = x \otimes x^*$}.

\begin{defn}
Let $\mathcal S = \{[x_j]\}_{j \in \Zn{N}} \subset \CPk{K-1}$ and $w: {\mathbb C}^K \rightarrow (0,1]$ be normalized so that $\sum\limits_{j \in \Zn{N}} w(x_j)=1$.
The pair $(\mathcal S, w)$ is a
 \textem{weighted complex projective $t$-design of dimension $K$}, or just \textem{weighted $t$-design}, if
$$
\sum\limits_{j \in \Zn{N}} w(x_j) \pi(x_j)^{\otimes t} = 
\int\limits_{\CPk{K-1}} \pi(x)^{\otimes t} d\mu(x)
= \left( \begin{array}{cc} K+t-1 \\ t \end{array} \right)^{-1} \Pi_{\text{sym}}^{(t)},
$$
where $\Pi_{\text{sym}}^{(t)}$ denotes the projection onto the totally symmetric subspace of $\left( \mathbb C^K \right)^{\otimes t}$ and $\mu$ denotes the unique unitarily invariant probability measure on $\CPk{K-1}$ induced by the Haar measure on the group of $K\times K$ unitaries.
\end{defn}

The following theorem characterizes weighted $t$-designs in terms of the Gramian of $\{x_j\}_{j \in {\mathbb Z}_N}$.
\begin{thm}[\cite{MR2337670}]\label{thchar2designs}
If $\mathcal S = \{[x_j]\}_{j \in \Zn{N}} \subset \CPk{K-1}$ and $w: {\mathbb C}^K \rightarrow (0,1]$, normalized so that $\sum\limits_{j \in \Zn{N}} w(x_j)=1$, then
the pair $(\mathcal S, w)$ is a weighted $t$-design if and only if 
\begin{equation}
\sum_{j, l \in \Zn{N}} w(x_j) w(x_l) |\langle x_j, x_l \rangle |^{2t} = \left( \begin{array}{cc} K+t-1 \\ t \end{array} \right)^{-1}
.\end{equation}
\end{thm}

In order to form a (weighted) 2-design, the projectors $\pi(x_j)$ are required to span the space of $K\times K$ matrices \cite[Theorem 4]{MR2269701}, 
which implies that $N\ge K^2$.
In addition, \cite[Theorem 4]{MR2269701} shows that a weighted 2-design of dimension $K$ has $K^2$ elements if and only if
it is formed by an equiangular tight frame of $K^2$ unit-norm vectors. Moreover, a given number of positive matrices
$\{A_j\}_{j=1}^N$ that span the space of complex $K\times K$ matrices and satisfy $\sum_{j=1}^N A_j = I$
are optimal for linear quantum state tomography from a measurement with $N$ outcomes 
if and only if the matrices are all of rank one and obtained from a weighted $2$-design $\{[x_j]\}_{j=1}^N$ according to
$A_j = \pi(x_j)$ \cite[Corollary 19]{MR2269701}.

Using the characterization, we show that OGFs obtained through the Singer construction yield such weighted $2$-designs.

\begin{thm}
If $\mathcal F = \{e_i\}_{i=1}^K \cup \mathcal F'$ is a frame for $\mathbb C^K$ 
formed by adjoining the canonical basis $\{e_i\}_{i=1}^K$ with  
an equiangular tight frame $\mathcal F'=\{h_j\}_{j=1}^M$ of $M=K^2-K+1$ flat vectors, then the pair $(\mathcal S, w)$ 
with the set $\mathcal S = \{[x]: x \in \mathcal F\}$ and the weights defined by
$$
w(e_i)  = 
\alpha = \frac{K^2 - K +1}{K (K^3 +1)} \, , \quad i \in \{1, 2, \dots, K\} \, ,
$$
and
$$
w(h_j) = \beta  = \frac{K}{K^3 +1} \, , \quad j \in {\mathbb Z}_M , 
$$
forms a weighted complex projective 2-design.
\end{thm}

\begin{proof}
Let $A = ( |\langle f_a, f_b \rangle |^4)_{j, l \in \Zn{N}}$, which is the matrix whose entries are the magnitudes of the entries of the Gramian of $\mathcal F$ raised to the fourth power. This matrix has the block form
$$
A = \left( \begin{array}{cc} I_K & B \\ B^* & A'   \end{array}   \right),
$$
where $B$ contains the magnitudes of the corresponding entries in the $K \times M$ cross-Gramian between the orthonormal basis vectors and ${\mathcal F}'$
raised to the fourth power, so all entries of $B$ are $\frac{1}{K^2}$. The submatrix $A'$ contains the fourth powers of the magnitudes of the entries of the Gramian of $\mathcal F'$, so it has ones along its diagonal and $\frac{(K-1)^2}{K^4}$ as its off-diagonal entries.
  Summing the weighted entries of $A$, 
we obtain
\begin{align*}
\sum_{a, b \in \Zn{N}} w(f_a) w(f_b) |\langle f_a, f_b \rangle |^{4} &=\alpha^2 K + \frac{2 \alpha \beta M}{K} + \beta^2\left(M + M (M-1) \frac{(K-1)^2}{K^4}\right) \\
& =  \left( \begin{array}{cc} K +1 \\ 2  \end{array}   \right)^{-1},
\end{align*}
where the last equality follows because $M = K^2 - K +1$, which completes the proof by 
 the characterization of $t$-designs
 for $t=2$
according to Theorem~\ref{thchar2designs}.
\end{proof}

\begin{cor}
If $q$ is a prime power, $K=q+1$ and $M=q^2+q+1=K^2-K+1$, then the cyclic frame $\mathcal F'=\{h_j\}_{j=1}^M$ associated
with the $(M,K,1)$-difference set constructed by Singer provides a complex weighted 2-design $({\mathcal S},w)$
with $\mathcal S$ and $w$ related to $\mathcal F = \{e_i\}_{i=1}^K \cup \mathcal F'$ as in the preceding theorem.
\end{cor}

Similarly, OGFs obtained through the picket fence construction also yield weighted 2-designs.

\begin{thm}\label{thm_picket_wtd2dsns}
If $\mathcal F = \{e_i\}_{i=1}^K \cup \mathcal F'$ is a frame for $\mathbb C^K$ formed by  adjoining the canonical basis $\{e_i\}_{i=1}^K$ with 
the union of 
flat unit-norm equiangular tight frames ${\mathcal F}'= \cup_{l=1}^{K-1} \{ h^{(l)}_j \}_{j \in \Zn{K+1} }$ for $\mathbb C^K$ that
are pairwise  mutually unbiased, so
$$
    | \langle h^{(l)}_j , h^{(l')}_{j'} \rangle |= \delta_{l,l'} \left( \frac 1 K + \delta_{j,j'} \frac{K-1}{K}\right) + (1-\delta_{l,l'}) \frac{1}{\sqrt K} ,
$$
then the pair $({\mathcal S}, w)$ with the set $\mathcal S = \{[x]: x \in \mathcal F\}$ and the weights defined by
$$
w(e_i)  = 
\alpha = \frac{1}{K (K+1)} \, , \quad i \in \{1, 2, \dots, K\} \, ,
$$
and
$$
w(h_j^{(l)}) = 
\beta  = \frac{K}{(K+1)(K^2-1)} \, , \quad j \in \Zn{K+1}, l \in \Zn{K-1},
$$
forms a weighted 
complex projective 2-design.
\end{thm}

\begin{proof}
Let $A = ( |\langle f_a, f_b \rangle |^4)_{a,b \in \Zn{N}}$, which is the matrix whose entries are the magnitudes of the entries of the Gramian of $\mathcal F$ raised to the fourth power. This matrix has the block form
$$
A = \left( \begin{array}{cc} I_K & B \\ B^* & A'   \end{array}   \right),
$$
where $B$ contains the magnitudes of the $K \times M$ cross-Gramian between the orthonormal basis and ${\mathcal F}'$
raised to the fourth power, so its entries are all $\frac{1}{K^2}$, and where $A'$ denotes the $M \times M$ matrix whose entries are the fourth powers of the magnitudes of the entries of the Gramian of ${\mathcal F}'$, which has 
one instance of $1$, $K$ instances of $\frac{1}{K^4}$ and $(K-2)(K+1)$ instances of $\frac{1}{K^2}$ occuring in each of its rows.
 Summing the weighted entries of $A$, 
we obtain
\begin{align*}
\sum_{a,b \in \Zn{N}} w(f_a) w(f_b) |\langle f_a, f_b \rangle |^{2t} &=\alpha^2 K + \frac{2 \alpha \beta M}{K} + \beta^2 M \left(1 + \frac{1}{K^3} + \frac{(K-2)(K+1)}{K^2} \right) \\
& =  \left( \begin{array}{cc} K +1 \\ 2  \end{array}   \right)^{-1},
\end{align*}
where the last equality follows because $M = K^2 - 1$.
Applying Theorem~\ref{thchar2designs} with $t=2$ then completes the proof.
\end{proof}

\begin{cor}
If $q$ is a prime power, $K=q$ and $M=q^2-1$, then the cyclic frame $\mathcal F'$ associated
with a picket fence sequence provides a complex weighted 2-design $({\mathcal S},w)$
with $\mathcal S$ and $w$ related to $\mathcal F = \{e_i\}_{i=1}^K \cup \mathcal F'$ as in the preceding theorem.
\end{cor}

\begin{proof}
Because $\mathcal F'$ is generated by some picket fence sequence $\{n_1,..., n_K \}$, we can write it as the union
$$
{\mathcal F}'= \cup_{l=1}^{K-1} \{ h^{(l)}_j \}_{j \in \Zn{K+1} },
$$
where 
$$
h^{(l)}_j = \left(e^{2 \pi i n_m [j(K-1) + l] / (K^2-1)} / \sqrt{K} \right)_{m=1}^K
$$
for each $j \in \Zn{K+1}$ and $l \in \{1,..., K-1\}$.  It follows from Theorem~\ref{thmpartdiffsetsgiveortho} that each subset $ \{ h^{(l)}_j \}_{j \in \Zn{K+1}}$ is equiangular, any two of such subsets are mutually unbiased, and these vectors are flat by construction.  It remains to show that each subset is a tight frame for $\mathbb C^K$.  It is straightforward to verify that for each $j \in \Zn{K+1}$ and $l \in \{1,..., K-1\}$, we can write 
$$
h_j^{(l)} = D_l h_j,
$$
where $D_l$ is the diagonal unitary $D_l = diag(e^{2 \pi i n_m l / (K^2-1)})_{m=1}^K$ and $h_j =\left(e^{2 \pi i n_m j / (K+1)}/\sqrt{K}\right)_{m=1}^K$, so it is enough to show that the set  $\{h_j \}_{j \in \Zn{K+1}}$ is a tight frame for $\mathbb C^K$.  
For $m \neq m'$, we have $n_m \not\equiv n_{m'} \mod (K+1)$;  otherwise, $n_m - n_{m'} = r (K+1)$ for some $r \in \mathbb Z$,
 which contradicts that $\{n_1,..., n_K\}$ is a picket fence sequence.  Thus, 
$$
\sum\limits_{j \in \Zn{K+1} } h_j \otimes h_j^* = \left( \sum\limits_{j \in \Zn{K+1}} e^{2 \pi i (n_a - n_b) j /K+1} / K  \right)_{a,b=1}^K = \frac{K+1}{K} I_K,
$$
where the second equality follows from the summation of consecutive powers of $(K+1)$th roots of unity on the off-diagonal entries.  In particular, this verifies that  $\{h_j \}_{j \in \Zn{K+1}}$ is a tight frame for $\mathbb C^K$, so the claim follows by applying Theorem~\ref{thm_picket_wtd2dsns}.
\end{proof}

\bibliography{ogfbib-all}
\bibliographystyle{plain}

%
%
%
%
%
%
%
%
%
%
%

\end{document}